\documentclass[conference]{IEEEtran}
\usepackage{enumitem}
\usepackage{multirow}
\usepackage{xcolor}

\usepackage{amsthm}
\usepackage{mathrsfs}

\usepackage{graphicx}

\ifCLASSINFOpdf
\else
\fi
\usepackage{amsmath}

\begin{document}
%
\title{Efficient Load Flow Techniques Based on Holomorphic Embedding for Distribution Networks}

\author{\IEEEauthorblockN{Majid Heidarifar, 
Panagiotis Andrianesis, Michael Caramanis}
\IEEEauthorblockA{Division of Systems Engineering\\
Boston University\\
Boston, Massachusetts 02446--8200\\
Email: mheidari, panosa, mcaraman@bu.edu}
}


%


\maketitle

\begin{abstract}
The Holomorphic Embedding Load flow Method (HELM) employs complex analysis to solve the load flow problem.
It guarantees finding the correct solution when it exists, and identifying when a solution does not exist.
The method, however, is usually computationally less efficient than the traditional Newton-Raphson algorithm, which is generally considered to be a slow method in distribution networks.
In this paper, we present two HELM modifications that exploit the radial and weakly meshed topology of distribution networks and significantly reduce computation time relative to the original HELM implementation.
We also present comparisons with several popular load flow algorithms applied to various test distribution networks.
\end{abstract}

%
\IEEEpeerreviewmaketitle

\section{Introduction}
Distribution networks are different from transmission systems in several aspects, including radial or weakly meshed structures, high $R/X$ ratios, non-transposed conductors, single or two-phase laterals, etc. Due to these inherent dissimilarities, conventional load flow methods, e.g., the Newton-Raphson (NR) method, are usually less efficient when applied to distribution networks relative to transmission system applications.

Taking advantage of the radial structure of distribution networks, the backward-forward sweep (BFS) algorithm of \cite{Shirmohammadi_EtAl-powersys1988} and the direct approach of \cite{Teng-PowerDelivery2003} are well-known for their computational efficiency.
In BFS, the forward sweep consists of a voltage update step starting from the slack node towards the far-end nodes, whereas the backward sweep is a current summation algorithm in the reverse direction. The direct approach introduces two matrices, namely the upper triangular Bus Injection to Branch Current (BIBC) matrix, and the lower triangular Branch Current to Bus Voltage (BCBV) matrix, in an iterative voltage update scheme.

Despite their computational efficiency, both the BFS and the direct approach do not guarantee finding the solution when it exists.
In particular, \cite{Araujo_EtAl-IJEPES2010} and \cite{Araujo_EtAl-IJEPES2018} demonstrate divergence issues of BFS and a better performance by a variant of NR, known as the current injection method (CIM), in the presence of constant impedance loads and under high loading conditions. Nevertheless, the convergence of NR-based methods is problematic, as well. They depend strongly on initialization, and they do not guarantee, in general, convergence to the correct solution. Moreover, divergence in the aforementioned methods is neither a necessary nor a sufficient condition for the solution's non-existence.

Aiming at addressing these issues, \cite{Trias-PESGM2012} proposes the Holomorphic Embedding Load flow Method (HELM) that employs analytic continuation of complex analytic functions. 
Analytic continuation methods, e.g., Pad\'{e} approximants representing nodal voltage functions in HELM, can evaluate a function beyond the radius of convergence of its respective power series. Using diagonal or near-diagonal Pad\'{e} approximants, which provide maximal analytic continuation, ensures the theoretical convergence of HELM. More specifically, HELM promises to find a solution, if it exists. Further, it identifies the  non-existence of a solution by detecting oscillations in the Pad\'{e} approximant sequence.

Despite the theoretically attractive properties of HELM, it is believed to be computationally expensive. HELM requires solving the recursive solution of a set of linear equations to obtain the coefficients of the voltage power series, followed by another set of linear equations, solved to obtain the coefficients of Pad\'{e} approximants. In fact, \cite{Rao_EtAl-PowerSys2016} shows worse performance compared to NR methods for transmission systems. The performance of HELM in distribution networks has not, however, been thoroughly investigated in the literature. To the best of our knowledge, a first attempt is made by \cite{Rao_EtAl-NAPS2016} in the context of network reduction, and by \cite{Ju_PrePrint2018}, which extends HELM to three-phase unbalanced systems.

In this paper, we exploit the radial and weakly meshed structure of distribution networks, and propose two HELM modifications which require less computational effort, while maintaining the theoretical convergence properties.
We implement the proposed HELM modifications on several test networks, and show that they achieve lower computation times compared with the original HELM implementation. 
We also present comparisons with other popular load flow methods, namely the BFS, the direct approach, the implicit Z-Bus method, and NR method, and we investigate the impact of different loading conditions and ZIP load models on computation times.

The remainder of the paper is structured as follows. In Section \ref{HELM}, we present a brief overview of HELM. In Section \ref{RadHELM}, we introduce our proposed HELM modifications. In Section \ref{SimRes}, we present and discuss the numerical results on several standard test cases. Lastly, we summarize our key findings in Section \ref{Conc} and provide directions for further research.

\section{Holomorphic Embedding Load Flow Method}
\label{HELM}


We consider a distribution network with $N+1$ nodes, where node $0$ is the ``slack node'' and all other nodes are of PQ-type in the set $\mathscr{N} = \{1,..,N\}$. The load flow equations are as follows:
\begin{equation}
    \sum_{j=0}^{N}{y_{ij}}V_{j} = \frac{S_{i}^{*}}{V_{i}^{*}}, \ \ \ \ i \in \mathscr{N},
    \label{loadflow}
\end{equation}
where $y_{ij}$ is the $ij$-th element of the ``$Y$-bus'' admittance matrix, $V_j$ is the complex-valued voltage at node $j$, and $S_i$ is the constant apparent power of load at the $i$-th node. The superscript $^*$ denotes the complex conjugate operator. Because the load flow equation \eqref{loadflow} is not holomorphic, \cite{Trias-PESGM2012} proposes to embed a complex-valued parameter $\alpha$, and obtain the following implicit holomorphic voltage function:
\begin{equation}
    \sum_{j=0}^{N}{y_{ij}}V_{j}(\alpha) = \alpha\frac{S_{i}^{*}}{V_{i}^{*}(\alpha^{*})}, \ \ \ \ i \in \mathscr{N}.
    \label{embedded}
\end{equation}

Note that at $\alpha=0$, a simple solution --- called germ --- can be found under the no-load, no-generation scenario, and the original load flow equations (\ref{loadflow}) can be recovered at $\alpha=1$. 
Following \cite{Trias-PESGM2012}, we replace  $V_{j}(\alpha)$ with its Maclaurin series (employing the germ), i.e.,
\begin{equation}
V_{j}(\alpha) = v_{j}[0]+v_{j}[1]\alpha+...+v_{j}[n]\alpha^{n}+...,  \ j \in \mathscr{N},
\label{VoltageSeries}
\end{equation}
and the reciprocal of the voltage  $V_i(a)$ with power series $W_i$:
\begin{equation}
    \frac{1}{V_{i}(\alpha)} = W_{i}(\alpha) = w_{i}[0]+w_{i}[1]\alpha+...+w_{i}[n]\alpha^{n}+..., \  i \in \mathscr{N},
    \label{reciprocal}
\end{equation}
 which implies that the coefficients of the voltage power series and its reciprocal at node $i \in \mathscr{N}$, $v_{i}$ and $w_{i}$, respectively, are related by their convolution:
\begin{equation}
    w_{i}[n] = \begin{cases}
    \dfrac{1}{v_{i}[n]},       & \quad n = 0,\\
    -\dfrac{\sum_{k=0}^{n-1}{w_i[k]v_i[n-k]}}{v_{i}[0]},  & \quad n \geq 1.
  \end{cases}
    \label{convolution}
\end{equation}

Substituting \eqref{VoltageSeries} and \eqref{reciprocal} into \eqref{embedded}, and equating the coefficients of $\alpha$ at both sides, yields \cite{Trias-PESGM2012}:
\begin{equation}
    \sum_{j=0}^{N}{y_{ij}}v_{j}[n] = {S_{i}^{*}}{w_{i}^{*}[n-1]}, \ \ \ \ i \in \mathscr{N}.
    \label{HELMCoeff}
\end{equation}

HELM requires solving \eqref{convolution} and \eqref{HELMCoeff} recursively, starting with the germ, i.e.,  $v_{i}[0], \ \forall i\in\mathscr{N}$. However, calculating the germ is not straightforward in case shunt admittances are included in the network model. This has motivated \cite{Rao_EtAl-PowerSys2016} to express the admittance matrix $\mathbf{Y}$ as:
\begin{equation}
    \mathbf{Y} = \mathbf{Y}^{s} + \mathbf{Y}^{sh},
    \label{Ybus}
\end{equation}
where $\mathbf{Y}^{s}$ includes the series part of the admittance matrix and $\mathbf{Y}^{sh}$ is a diagonal matrix of the shunt admittances. Using \eqref{Ybus}, \eqref{embedded} becomes \cite{Rao_EtAl-PowerSys2016}:
\begin{equation}
    \sum_{j=0}^{N}{y^{s}_{ij}}V_{j}(\alpha) = \alpha\frac{S_{i}^{*}}{V_{i}^{*}(\alpha^{*})}-\alpha y^{sh}_{i}V_{i}(\alpha), \ \ i \in \mathscr{N},
    \label{embeddedShunt}
\end{equation}
where $y^{s}_{ij}$ is the $ij$-th element of the $\mathbf{Y}^{s}$ matrix, and $y^{sh}_{i}$ is the $i$-th diagonal element of the $\mathbf{Y}^{sh}$ matrix.
From \eqref{embeddedShunt}, the germ calculation becomes straightforward and is given by $v_{i}[0]=V_{0}, \ \forall i \in \mathscr{N}$, where $V_0$ is the slack-node voltage.
Substituting \eqref{VoltageSeries} and \eqref{reciprocal} into \eqref{embeddedShunt}, and equating the coefficients of $\alpha$ at both sides, yields \cite{Rao_EtAl-PowerSys2016}:
\begin{equation}
    \sum_{j=0}^{N}{y^{s}_{ij}}v_{j}[n] = {S_{i}^{*}}{w_{i}^{*}[n-1]}-{y^{sh}_{i}}{v_{i}[n-1]},  \ \ i \in \mathscr{N}.
    \label{HELMCoeffShunt}
\end{equation}


Hence, HELM recursively solves \eqref{convolution} and \eqref{HELMCoeffShunt} to obtain higher-order coefficients of the voltage power series at each step.
The voltage power series may not, however, be converging as it is often the case. Therefore, \cite{Trias-PESGM2012} proposes to evaluate the  voltage at PQ node $i \in \mathscr{N}$ using a rational function ---  the Pad\'{e} approximant --- described as:
\begin{equation}
    V_{i}(\alpha) \approx [L/M]_{V_{i}(\alpha)}=\frac{\zeta_{i}[0]+\zeta_{i}[1]\alpha+...+\zeta_{i}[L]\alpha^{L}}{1+\beta_{i}[1]\alpha+...+\beta_{i}[M]\alpha^{M}},
    \label{PadeLM}
\end{equation}
which is often used as an analytic continuation method to evaluate a function outside the radius of convergence of its power series but within the function's domain.
The commonly used method to calculate the coefficients of the polynomials in \eqref{PadeLM} is the so-called matrix method \cite{Baker_EtAl-Book1996}, in which the denominator coefficients $\beta$ at each node can be obtained by solving a dense linear system of equations \cite{Baker_EtAl-Book1996}, thus time-consuming and prone to errors as the order of the Pad\'{e} approximant increases \cite{Rao_EtAl-IJEPES2018}. The numerator coefficients $\zeta$ can then be found by a back-substitution.

An alternative method is employed in \cite{Rao_EtAl-IJEPES2018} --- the Eta method, which always results in a diagonal Pad\'{e} approximant, thereby retaining the convergence promises of HELM, and which has reportedly better performance in finding a converging voltage solution.
The Eta method obtains a converging voltage power series based on a two dimensional array called the $\eta$ table. Therefore, it allows the evaluation of the nodal voltage at node $i \in \mathscr{N}$ and recursive step $n$, denoted by $V_{i}^{(n)}$.
Its drawback is that it only yields the solution at $\alpha=1$, as opposed to the matrix method that yields a solution as a function of $\alpha$.
Nonetheless, this drawback does not affect load flow problems, since, as already mentioned, the solution at $\alpha=1$ suffices to recover the original load flow equations.


In what follows, we present an outline of the recursive algorithm that describes the original HELM encompassing the Eta method, which we use for comparison purposes.
\begin{enumerate}[label=\textbf{{Step}{{ \arabic*}}:}]
    \setlength{\itemindent}{1.5em}
    \item Set $n=0$, and $v_{i}[0]$ = $V_{0}, \ \forall i \in \mathscr{N}$.
    \item Calculate $w_{i}[n], \ \forall i \in \mathscr{N}$, using \eqref{convolution}.
    \item Set $n=n+1$. Calculate RHS of \eqref{HELMCoeffShunt}.
    \item Obtain $v_i[n], \ \forall i \in \mathscr{N}$, solving \eqref{HELMCoeffShunt}.
    \item Evaluate $V_i^{(n)}, \ \forall i \in \mathscr{N}$, using the Eta method. Check convergence (for tolerance $\epsilon$): If $|V_i^{(n)}-V_i^{(n-1)}|< \epsilon, \forall i \in \mathscr{N}$, then stop; otherwise, recursively apply steps 2--5.
\end{enumerate}

\section{The Proposed Load Flow Methods}
\label{RadHELM}
In this section, we present two HELM modifications that are tailored to distribution networks. They involve two alternative methods for solving \eqref{HELMCoeffShunt} in Step 4 of the original HELM, for a radial or weakly meshed topology. Our work is inspired by two popular load flow algorithms in distribution networks, namely the BFS algorithm and the direct approach of \cite{Teng-PowerDelivery2003}.

The BFS is an efficient algorithm for radial distribution networks. At each iteration $k$, it calculates the nodal voltages and is described as follows:
\begin{gather}
\mathbf{\tilde{A}} \mathbf{I}_{b}^{(k)}=\mathbf{\tilde{I}}\big(\mathbf{\tilde{V}}^{(k-1)}\big), \label{Backward}\\
\mathbf{Y}_{b} \mathbf{\tilde{A}}^{T} ( \mathbf{\tilde{V}}^{(k)} - V_{0} \mathbf{1}_{N} )=\mathbf{I}_{b}^{(k)}, \label{Forward}
\end{gather}
where $\mathbf{\tilde{V}}$, $\mathbf{\tilde{I}}$, and $\mathbf{I}_{b}$ are $N\times 1$ vectors denoting nodal voltages at PQ nodes, current injections at PQ nodes, and branch currents, respectively. Note that shunt admittances are modeled in the current injection vector $\mathbf{\tilde{I}}$. $\mathbf{Y}_{b}$ is a $N\times N$ diagonal matrix whose elements correspond to branch admittances, and $\mathbf{\tilde{A}}$ is obtained from the node to branch incidence matrix $\mathbf{A}$, partitioned as follows:
\begin{equation}
\mathbf{A}=
\begin{pmatrix}
\mathbf{a}_{0}^{T} \\
\mathbf{\tilde{A}} 
\end{pmatrix},
\label{AMat}
\end{equation}
where $\mathbf{a}_{0}^{T}$ and $\mathbf{\tilde{A}}$ are the rows of $\mathbf{A}$ associated with the slack and PQ nodes, respectively. Note that the rhs of \eqref{Backward} explicitly shows the dependence of nodal current injection $\mathbf{\tilde{I}}$ on nodal voltages at the previous iteration $\mathbf{\tilde{V}}^{(k-1)}$. In a radial network, obtaining the vector of branch currents $\mathbf{I}_{b}$ from \eqref{Backward} is equivalent to a \textit{backward sweep}, whereas utilizing $\mathbf{I}_{b}$ in \eqref{Forward} to solve for voltages is equivalent to a \textit{forward sweep} \cite{Zhang_EtAl-PowerSys1997}. A backward sweep calculates the branch currents by traveling backward from the far-end nodes to the slack node, whereas a forward sweep updates nodal voltages by traveling forward from the slack node to the far-end nodes. \eqref{Backward} and \eqref{Forward} are solved iteratively until the difference in nodal voltages at successive iterations is less than a tolerance.

The direct approach \cite{Teng-PowerDelivery2003} is applicable to both radial and weakly meshed systems and is described as follows:
\begin{gather}
\mathbf{\mathbf{\tilde{V}}}^{(k)} =V_{0} \mathbf{1} + (\mbox{DLF})\times\mathbf{\tilde{I}}\big(\mathbf{\tilde{V}}^{(k-1)}\big), \label{DirectApp}
\end{gather}
where $\mbox{DLF}$ is a constant matrix and $\mathbf{1}$ is a $N\times 1$ vector of all ones. Details on how to construct $\mbox{DLF}$ can be found in \cite{Teng-PowerDelivery2003}. 


Observing \eqref{HELMCoeffShunt}, we identify an interesting interpretation, presented in Lemma \ref{Lemma1}.

\newtheorem{lemma}{Lemma}
\begin{lemma}
At the $n$-th voltage power series coefficients calculation, \eqref{HELMCoeffShunt} is equivalent to \eqref{LF} below which describes a load flow problem of a network without shunt elements and only constant current type  loads/injections:
\begin{equation}
    \mathbf{Y}\mathbf{V} = \mathbf{I}.
    \label{LF}    
\end{equation}

Note that $\mathbf{I}$ is the vector of current load/injections and $\mathbf{V}$ the vector of nodal voltages (including the slack bus).
\label{Lemma1}
\end{lemma}
\begin{proof}
At the $n$-th voltage power series coefficients calculation, the rhs of \eqref{HELMCoeffShunt} is constant, i.e., $I_{i} = {S_{i}^{*}}{w_{i}^{*}[n-1]}-{y^{sh}_{i}}{v_{i}[n-1]}, \ \forall i \in \mathscr{N}$. At the lhs, setting $y_{ij} = y^{s}_{ij}$, indicates that the corresponding network has a bus admittance matrix $\mathbf{Y}^s$, i.e., without shunt elements.
\end{proof}

The network introduced in Lemma \ref{Lemma1} has some attractive properties.

\begin{lemma}
The load flow problem of radial networks described in Lemma \ref{Lemma1} can be solved in a single iteration of the BFS algorithm (one backward and one forward sweep).
\label{Lemma2}
\end{lemma}
\begin{proof}
The $\mathbf{Y}$-Bus admittance matrix can be partitioned as:
\begin{equation}
\begin{pmatrix}
  y_{0} & \mathbf{y}^{T}\\
  \mathbf{y} & \mathbf{\tilde{Y}} 
 \end{pmatrix}
\begin{pmatrix}
  V_{0}\\
  \mathbf{\tilde{V}} 
 \end{pmatrix}
=
\begin{pmatrix}
  I_{0}\\
  \mathbf{\tilde{I}} 
 \end{pmatrix}
,
 \label{YsMat}
\end{equation}
where $\mathbf{y}$ is a $N\times 1$ vector describing the mutual admittances between the slack  and PQ nodes,
${y}_{0}$ is the self admittance of the slack bus and $\mathbf{\tilde{Y}}$ is an $N\times N$ matrix.
Since the $\mathbf{Y}$-Bus described in Lemma \ref{Lemma1} does not include shunt admittances, we can equivalently express it as:
\begin{equation}
\mathbf{Y}= \mathbf{A} \mathbf{Y}_{b} \mathbf{A}^{T}.
 \label{Ys}
\end{equation}

Using \eqref{AMat} and \eqref{YsMat}, we get from \eqref{Ys}:
\begin{equation}
\mathbf{y} = \mathbf{\tilde{A}} \mathbf{Y}_{b} \mathbf{a}_{0}, \ \ \ \ \
\mathbf{\tilde{Y}} = \mathbf{\tilde{A}} \mathbf{Y}_{b} \mathbf{\tilde{A}}^{T}.
\label{Ys1}
\end{equation}

Expressing \eqref{LF} for PQ-type nodes using \eqref{YsMat} yields:
\begin{equation}
 \mathbf{y} V_{0}+ \mathbf{\tilde{Y}} \mathbf{\tilde{V}} = \mathbf{\tilde{I}}.
 \label{Ys2}
\end{equation}

Substituting \eqref{Ys1} into \eqref{Ys2}, we get:
\begin{equation}
\mathbf{\tilde{A}} \mathbf{Y}_{b} ( \mathbf{a}_{0} V_{0} + \mathbf{\tilde{A}}^{T} \mathbf{\tilde{V}})=\mathbf{\tilde{I}}.
\label{Ys3}
\end{equation}

Note that sum of the elements in each column of $\mathbf{A}$ should be zero, which implies that $\mathbf{a}^{T}_{0} = - \mathbf{\tilde{A}}^{T} \mathbf{1}$. Therefore, \eqref{Ys3} can be written as:
\begin{equation}
\mathbf{\tilde{A}} \mathbf{Y}_{b} \mathbf{\tilde{A}}^{T}  ( \mathbf{\tilde{V}} - V_{0} \mathbf{1} )=\mathbf{\tilde{I}}, \label{FinalDeriv}
\end{equation}
which can be expressed equivalently as:
\begin{gather}
\mathbf{\tilde{A}} \mathbf{I}_{b}=\mathbf{\tilde{I}}, \label{Backward1}\\
\mathbf{Y}_{b} \mathbf{\tilde{A}}^{T} ( \mathbf{\tilde{V}} - V_{0} \mathbf{1} )=\mathbf{I}_{b}, \label{Forward1}
\end{gather}
representing a single backward and forward sweep, as discussed for \eqref{Backward} and \eqref{Forward}. We note that a similar argument was made in \cite{Zhang_EtAl-PowerSys1997} for a different purpose.
\end{proof}

\begin{lemma}
The load flow problem of radial or weakly meshed networks described in Lemma \ref{Lemma1} can be solved in a single iteration of the direct approach.
\label{Lemma3}
\end{lemma}
\begin{proof}
A single iteration of the direct approach is described as:
\begin{gather}
\mathbf{\tilde{V}} = V_{0} \mathbf{1} + (\mbox{DLF})\mathbf{\tilde{I}}. \label{DirDV}
\end{gather}

A direct comparison of \eqref{DirDV} with \eqref{FinalDeriv} yields:
\begin{gather}
\mbox{DLF} = (\mathbf{\tilde{A}} \mathbf{Y}_{b} \mathbf{\tilde{A}}^{T})^{-1}. \label{DLFDer}
\end{gather}

Therefore, \eqref{LF} can be solved using a single iteration of the direct approach described by \eqref{DirDV}.
 \end{proof}

We note that, for a radial network, $\mbox{DLF}=\mbox{BCBV}\times\mbox{BIBC}$ \cite{Teng-PowerDelivery2003}. Therefore, \eqref{DirDV} can be expressed equivalently as:
\begin{gather}
\mathbf{I}_{b}=(\mbox{BIBC})\mathbf{\tilde{I}}, \label{BIBC}\\
\mathbf{\tilde{V}} = V_{0} \mathbf{1}  + (\mbox{BCBV})\mathbf{I}_{b}. \label{BCBV}
\end{gather}

\newtheorem{corol}{Corollary}
\begin{corol}
For a radial network, the backward \eqref{Backward} and forward \eqref{Forward} sweeps are expressed equivalently in a matrix form using the direct approach, with:
\begin{gather}
\mbox{BIBC} = \mathbf{\tilde{A}}^{-1}, \label{BackBIBC}\\
\mbox{BCBV} = (\mathbf{Y}_{b} \mathbf{\tilde{A}}^{T})^{-1}, \label{ForwBCBV}
\end{gather}
where $\mathbf{\tilde{A}}$ is an $N\times N$ matrix.
\label{corol1}
\end{corol}
\begin{proof}
Corollary \ref{corol1} is derived by direct comparison of \eqref{BIBC} and \eqref{BCBV} with \eqref{Backward1} and \eqref{Forward1}, respectively.
\end{proof}

\newtheorem{prop}{Proposition}
\begin{prop}
The solution of \eqref{HELMCoeffShunt} for radial or weakly meshed networks is obtained by performing a single iteration of the BFS (radial) or the direct approach (radial/weakly meshed) methods.
\label{prop1}
\end{prop}
\begin{proof}
The proof of Proposition \ref{prop1} is straightforward, using \emph{(i)} Lemma \ref{Lemma1} for a network with $y_{ij} = y^{s}_{ij}$, and constant currents at the $n$-th voltage power series coefficients calculation, $I_{i} = {S_{i}^{*}}{w_{i}^{*}[n-1]}-{y^{sh}_{i}}{v_{i}[n-1]}, \ \forall i \in \mathscr{N}$, \emph{(ii)} Lemma \ref{Lemma2} for radial networks, and \emph{(iii)} Lemma \ref{Lemma3} for both radial and weakly meshed networks.
\end{proof}

In what follows, we introduce the two HELM modifications that modify Step 4, solving \eqref{HELMCoeffShunt} based on Proposition \ref{prop1}. 
The first algorithm, referred to as \textit{S-HELM}, uses the BFS algorithm and is suitable for radial networks.
The second algorithm, referred to as \textit{D-HELM}, uses the direct approach of \cite{Teng-PowerDelivery2003} and can be applied to both radial and weakly meshed networks. We summarize the modified steps below.

The S-HELM algorithm modifies Step 4 as follows:

\begin{enumerate}[label=\textbf{{Step}{{ \arabic*}}-S.a:}]
    \setlength{\itemindent}{3.2em}
  \setcounter{enumi}{3}
    \item Calculate branch currents using the nodal injections given by Step 3 of HELM, employing a backward current summation scheme.
\end{enumerate}
\begin{enumerate}[label=\textbf{{Step}{{ \arabic*}}-S.b:}]
    \setlength{\itemindent}{3.3em}
  \setcounter{enumi}{3}
    \item Solve for $v_i[n], \ \forall i \in \mathscr{N}$ using branch currents calculated in the previous step, employing a forward voltage update scheme.
\end{enumerate}

The D-HELM algorithm modifies Step 4 of the original HELM algorithm as follows:

\begin{enumerate}[label=\textbf{{Step}{{ \arabic*}}-D:}]
    \setlength{\itemindent}{2.7em}
  \setcounter{enumi}{3}
    \item Solve \eqref{DirDV} for $v_i[n]=V_{i}, \ \forall i \in \mathscr{N}$, where $I_i$ is given by Step 3 of HELM.
\end{enumerate}


\section{Numerical Results}
\label{SimRes}
We tested the proposed algorithms on several radial and weakly meshed distribution networks and we compared the results with other methods, namely the original HELM (using the Eta method), the BFS \cite{Shirmohammadi_EtAl-powersys1988}, the direct approach, \cite{Teng-PowerDelivery2003}, the implicit Z-bus \cite{Chen_EtAl-powerdelivery1991}, and NR. We modeled the methods in MATLAB v. 9.4 and used a Dell XPS i7 at 1.8 GHZ CPU with 16 GB RAM for obtaining the numerical results.




The test networks included IEEE 13, 18, 33, 37, 69, 123, 141 and 8500 node radial test systems. The data for 18, 33, 69, and 141 bus networks can be found in MATPOWER \cite{Zimmerman_EtAl-PowerSys1999}. We used the single phase equivalents of IEEE 13, 37 and 123 node networks derived by \cite{Bazrafshan_EtAl-PowerSys2018}, and we also obtained the single phase equivalent of the IEEE 8500 node distribution network that includes about 2500 nodes. The convergence tolerance was $\epsilon = 10^{-6}$. The maximum error in nodal voltage magnitudes compared with Implicit Z-bus, BFS, the direct approach, and NR was observed to be less than the tolerance, thus verifying the accuracy of the proposed HELM modifications.

In Table \ref{Tab:CompTime}, we present the computational times for all methods and test networks. They include the time required to run only the main loop of the methods; they do not include pre-processing time, e.g., branch ordering in BFS or LU factorization in original HELM. In order to derive accurate results, we ran the main loop of each method for $100,000$ times and obtained the mean computation times. 
The results indicate that at least one of the proposed HELM modifications (S-HELM and D-HELM) outperforms the original HELM --- in fact in all but one networks both modifications outperform the original HELM. 
D-HELM appears more efficient in smaller networks outperforming both original HELM and S-HELM, but achieves similar times with HELM for the large network.
For the latter network, S-HELM performs better than D-HELM and original HELM.
Overall, the BFS algorithm and the direct approach perform better, followed by the implicit Z-bus. NR is generally slower than HELM except for the 8500-node network. Interestingly but unsurprisingly, comparing Implicit Z-bus with BFS and the direct approach yields similar results to comparing HELM with S-HELM and D-HELM; note that HELM and Implicit Z-bus employ LU factorization, whereas S-HELM and D-HELM are based on the BFS and the direct approach, respectively.


\begin{table}
\renewcommand{\arraystretch}{1.3}
\caption{Load Flow Computation Time (in Milliseconds) on Several IEEE Radial Distribution Test Systems}
\label{Tab:CompTime}
\centering
\begin{tabular}{c|c|c|c|c|c|c|c|c}
    \hline
    \multirow{2}{*}{Methods}&\multicolumn{8}{c}{Test Systems}\\
    \cline{2-9}
    & 13 & 18 & 33 & 37 & 69 & 123 & 141 & 8500\\
    \hline
    \hline
    HELM & 0.30 & 0.41 & 0.56 & 0.46 & 0.89 & 1.50 & 1.63 & 58.3\\
    \hline
    S-HELM & 0.24 & 0.36 & 0.48 & 0.40 & 0.80 & 1.30 & 1.44 & 51.5\\
    \hline
    D-HELM & 0.22 & 0.31 & 0.41 & 0.32 & 0.68 & 1.19 & 1.32 & 60.6\\
    \hline
    Impl.Z & 0.14 & 0.21 & 0.19 & 0.12 & 0.27 & 0.32 & 0.44 & 13.0\\
    \hline
    BFS & 0.05 & 0.10 & 0.08 & 0.08 & 0.13 & 0.22 & 0.21 & 7.92\\
    \hline
    Direct & 0.06 & 0.11 & 0.08 & 0.07 & 0.11 & 0.18 & 0.17 & 19.1\\
    \hline
    NR & 0.39 & 0.45 & 0.64 & 0.82 & 1.05 & 2.55 & 2.00 & 38.5 \\
    \hline
\end{tabular}
\end{table}

Table \ref{Tab:Percent} presents the average percentage of computation time spent on each step of HELM. We observe that the percentage of Steps 2,3 and 4 decrease with the size of the network, whereas the percentage of Step 5 increases with the size of the network. 

The results in Table \ref{Tab:CompTime} and \ref{Tab:Percent} indicate that the proposed modifications achieve computational savings that range from 53\% to 92\% in Step 4 of HELM, and from 12\% to 30\% overall (including all HELM steps).

Further, as it is also discussed in \cite{Rao_EtAl-PowerSys2016}, since the voltage evaluation and convergence check performed in Step 5 is not used as input in Steps 2--4, one can proceed from Step 4 to Step 2 while in parallel checking convergence in Step 5. If convergence is reached, then algorithm terminates, otherwise convergence is checked again at the next step.
Taking into account the parallel implementation of Step 5, the computational savings of the proposed modifications (including all steps) range from 24\% to 50\% overall.

\begin{table}
\renewcommand{\arraystretch}{1.3}
\caption{Average Percentage of Computation Time Spent on each Step of HELM in Several IEEE Radial Distribution Test Systems}
\label{Tab:Percent}
\centering
\begin{tabular}{c|c|c|c|c|c|c|c|c}
    \hline
    \multirow{2}{*}{Steps}&\multicolumn{8}{c}{Test Systems}\\
    \cline{2-9}
    & 13 & 18 & 33 & 37 & 69 & 123 & 141 & 8500\\
    \hline
    \hline
    2, 3 & 32\% & 37\% & 11\% & 21\% & 5\% & 10\% & 5\% & 4\% \\
    \hline
    4 & 40\% & 31\% & 34\% & 33\% & 29\% & 27\% & 22\% & 22\% \\
    \hline
    5 & 28\% & 32\% & 55\% & 46\% & 66\% & 63\% & 73\% & 74\% \\
    \hline
\end{tabular}
\end{table}



We also considered weakly meshed variations of the IEEE 18, 33 and 69 node networks. Table \ref{Tab:CompTimeWM} shows that D-HELM requires less computation time compared with HELM, and it also outperforms the NR method. However, we acknowledge that more testing is required for weakly meshed networks.

\begin{table}
\renewcommand{\arraystretch}{1.3}
\caption{Load Flow Computation Time (in Milliseconds) on Weakly Meshed Variants of IEEE Distribution Test Systems}
\label{Tab:CompTimeWM}
\centering
\begin{tabular}{c|c|c|c}
    \hline
    \multirow{2}{*}{Methods}&\multicolumn{3}{c}{Test Systems}\\
    \cline{2-4}
    & 18w & 33w & 69w \\
    \hline
    \hline
    HELM & 0.42 & 0.43 & 0.52 \\
    \hline
    D-HELM & 0.27 & 0.32 & 0.39 \\
    \hline
    Impl.Z & 0.18 & 0.16 & 0.19 \\
    \hline
    Direct & 0.10 & 0.07 & 0.14 \\
    \hline
    NR & 0.53 & 0.73 & 1.21 \\
    \hline
\end{tabular}
\end{table}

Lastly, we evaluated the performance of the proposed modifications under different loading conditions. We selected the IEEE 123-node radial distribution network, which contains a complete ZIP load model and considered medium and high loading conditions. The load factors for constant power, constant current, and constant impedance loads were 4, 20, and 40, respectively, for the medium loading conditions, and 7, 50, and 60, respectively for the high loading conditions. We present the results in Table \ref{Tab:LoadModel}.  
We observe that the proposed modifications outperform the original HELM in all scenarios. Further, while the BFS, the direct approach and the NR method diverge under certain loading conditions, HELM and its proposed modifications manage to find a solution.


\begin{table}

\renewcommand{\arraystretch}{1.3}
\caption{Load Flow Computation Time (in Milliseconds) for IEEE 123 Node Distribution Network in Different Loading Scenarios}
\label{Tab:LoadModel}
\centering
\begin{tabular}{c|c|c|c|c|c|c}
    \hline
    \multirow{2}{*}{Methods}&\multicolumn{2}{c|}{Const. Power}&\multicolumn{2}{c|}{Const. Current}&\multicolumn{2}{c}{Const. Impedance}\\
    \cline{2-7}
    & Medium & High & Medium & High & Medium & High\\
    \hline
    \hline
    HELM  & 2.5 & 13.3 & 2.42 & 10.1 & 2.96 & 4.86\\
    \hline
    S-HELM  & 2.2 & 12.7 & 2.24 & 9.64 & 2.75 & 4.60\\
    \hline
    D-HELM  & 2.0 & 12.2 & 1.95 & 9.05 & 2.45 & 4.21\\
    \hline
    Impl.Z  & 0.58 & 1.35 & 0.40 & 0.50 & 0.40 & 0.46\\
    \hline
    BFS  & 0.46 & 1.05 & 0.37 & 0.45 & 1.18 & Div\\
    \hline
    Direct  & 0.29 & 0.65 & 0.23 & 0.28 & 0.71 & Div\\
    \hline
    NR  & 2.45 & 3.12 & Div & Div & 3.06 & 3.06 \\
    \hline
\end{tabular}
\end{table}

\section{Conclusions and Further Research}
\label{Conc}
This paper presents two HELM modifications that exploit radial and weakly meshed structure of distribution networks and require less computation effort. Numerical experimentation demonstrates overall time savings of up to 30\% on IEEE radial distribution test cases. Furthermore, the proposed HELM modifications are shown to be robust against different loading types and conditions. In our future research, we aim at (\emph{i}) further investigating the performance in weakly meshed networks, (\emph{ii}) extending the proposed modifications to three-phase unbalanced systems, and (\emph{iii}) performing extensive numerical comparisons and sensitivity analysis with respect to network parameters and loading conditions.




%
\bibliographystyle{IEEEtran}
\bibliography{IEEEabrv,conf}

\end{document}